\documentclass[11pt]{article}
\usepackage[utf8]{inputenc}
\usepackage{amsmath,amssymb,amsthm}
\usepackage{mathtools}
\usepackage{fancyhdr}
\usepackage{xcolor}
\usepackage{geometry}
\usepackage{enumitem}
\usepackage{thmtools}
\usepackage{thm-restate}
\usepackage{hyperref}
\usepackage{parskip}

\newtheorem{theorem}{Theorem}[section]

\newtheorem{lemma}[theorem]{Lemma}

\newtheorem*{claim*}{Claim}

\theoremstyle{definition}

\newtheorem{example}[theorem]{Example}

\newtheorem{remark}[theorem]{Remark}

\DeclareMathOperator{\range}{range}

\def\deg{\textsf{deg}}

\begin{document}

\title{Alternative Entropy Bounds for Perfect Matchings in Bipartite Graphs}
\author{
Yusong Du\thanks{Institute of Mathematical Sciences, ShanghaiTech University. ~Email: {\tt duys2025@shanghaitech.edu.cn}} \and
Boqing Xue\thanks{Institute of Mathematical Sciences, ShanghaiTech University. ~Email: {\tt xuebq@shanghaitech.edu.cn}}
}
\date{}
\maketitle

\begin{abstract}
We refine Radhakrishnan's entropy proof of the Br\'egman--Minc bound by introducing a terminal-set framework in which selected vertices are revealed last. This gives new degree-sensitive upper bounds for the number of perfect matchings in bipartite graphs and an explicit formula for single-vertex terminal sets. The bounds recover the standard
Br\'egman equality family and improve the estimate for certain nonuniform degree sequences. We also obtain a $C_4$-free refinement complementary to the edge-count bound of Araujo, Balogh and Wang.
\end{abstract}

\section{Introduction}\label{section1}

Entropy methods have become a useful tool in combinatorics, especially for enumeration problems where direct counting is difficult. In this paper we use entropy to study upper bounds on the number of perfect matchings in bipartite graphs.

A perfect matching of a graph $G=(V,E)$ is a subset $\sigma \subseteq E$ such that every vertex of $G$ is incident with exactly one edge of $\sigma$. We denote the set of all perfect matchings of $G$ by $\mathcal M_{\mathrm{perf}}(G)$.

A classical theorem of Br\'egman gives a sharp upper bound for the number of perfect matchings in a bipartite graph in terms of the degrees on one side.

\begin{theorem}[Br\'egman]
Let $G$ be a bipartite graph with vertex classes $V=\{ v_1,\cdots,v_n\}$ and $W = \{w_1,\cdots,w_n\}$. Let $d_i:=\deg(v_i)\geq 1$ for $i=1,2,\ldots,n$. Then
\[
|\mathcal{M}_{\mathrm{perf}}(G)| \le \prod_{i=1}^n (d_i!)^{\frac{1}{d_i}} .
\]
\end{theorem}

This theorem was conjectured by Minc \cite{minc1963upper} and proved by Br\'egman \cite{bregman1973some}. Schrijver \cite{schrijver1978short} gave a short proof, while Alon and Spencer \cite{alon2016probabilistic} obtained a related proof by analyzing a randomized procedure. Later, Radhakrishnan \cite{Radhakrishnan1997AnEP} gave an elegant entropy proof of Br\'egman's theorem.

Br\'egman's theorem was also extended to general, not necessarily bipartite, graphs by Kahn and Lov\'asz, and related proofs and rediscoveries were given by Alon and Friedland~\cite{alon2008maximum},  Egorychev~\cite{Egorychev}, Friedland~\cite{friedland2008upper} and Cutler and Radcliffe~\cite{cutler2011entropy}. Cutler and Radcliffe's proof also uses entropy.

Br\'egman's bound is best possible in general. Indeed, if $d$ divides $n$ and the graph $G_{n,d}$ is the disjoint union of $n/d$ copies of the complete bipartite graph $K_{d,d}$, then
\[
|\mathcal M_{\mathrm{perf}}(G)|=(d!)^{\frac{n}{d}},
\]
which is exactly the value of Br\'egman's upper bound.

On the other hand, Br\'egman's bound can be far from the truth for some graphs. Consider the following bipartite graph $G_0$ with vertex classes $V=\{v_1,\ldots,v_n\}$ and $W=\{w_1,\ldots,w_n\}$. Let $v_n$ be adjacent to every vertex of $W$, and for $1\le i\le n-1$, let $v_i$ be adjacent only to $w_i$. Then $G_0$ has exactly $1$ perfect matching. However, Br\'egman's theorem gives $|\mathcal M_{\mathrm{perf}}(G_0)| \le (n!)^{1/n}$. Thus the contribution of the high-degree vertex $v_n$ is substantially
overestimated.

Our goal is to refine the entropy argument by forcing selected vertices of $V$ to be revealed last. This modification reduces the entropy contribution of vertices whose matching choices are largely determined by the rest of the graph. The first consequence is the following closed-form bound, which corresponds to the case where one vertex is revealed last.

\begin{theorem}\label{main_thm_his1}
Let $G$ be a bipartite graph with vertex classes $V=\{ v_1,\cdots,v_n\}$ and $W = \{w_1,\cdots,w_n\}$. Denote $d_i=\deg(v_i)$ and $D_i=\deg(w_i)$ for $i=1,2,\ldots, n$. Assume that $1\leq d_1 \le d_2 \le \cdots  \le d_n$ and $1\leq D_1 \le D_2 \le \cdots \le D_n$. Let $\ell:=\min\{i:\, d_i>1\}$, with the convention that $\ell=n+1$ if all $d_i$ are equal to $1$. Then
\[
|\mathcal M_{\mathrm{perf}}(G)| \le \prod_{i=\ell}^{p} (d_i!)^{\frac{1}{d_i-1}} \prod_{i=p+1}^{n-1} (d_i!)^{\frac{1}{d_i}},
\]
where $p= \min\{D_n+\ell-2, n-1\}$.
\end{theorem}

Here and throughout the paper, an empty product is interpreted as $1$. Moreover, in the situation that $d_i=1$ for all $1\leq i\leq n$, we interpret $\ell=n+1$ and the asserted bound is $1$. The above theorem uses the $n-1$ smallest degrees in $V$ and the maximum degree in $W$. 

For the graph $G_0$ above, we have $d_1=\cdots=d_{n-1}=1$, $d_n=n$ and $D_n=2$. Thus $\ell=n$ and $p=n-1$. So Theorem~\ref{main_thm_his1} gives $|\mathcal M_{\mathrm{perf}}(G_0)|\le 1$, which is sharp.

The bound also preserves the standard equality example $G_{n,d}$ for Br\'egman’s theorem. In this example, assuming $d\geq 2$ without loss of generality, one has $\ell=1$, $D_n=d$, $p=d-1$, and Theorem~\ref{main_thm_his1} gives
\[
|\mathcal M_{\mathrm{perf}}(G_{n,d})| \le \prod_{i=1}^{d-1}(d!)^{\frac{1}{d-1}}
\prod_{i=d}^{n-1}(d!)^{\frac{1}{d}}
=
(d!)^{\frac{n}{d}},
\]
again matching the exact number of perfect matchings.

We now state the general form of our bound. Let
\[
H_t(d):=\frac{1}{d-t}\log\frac{d!}{t!},
\qquad 0\le t\le d-1,
\]
and for $0\le t\le d-2$, define the marginal increment
\[
\Delta_t(d)
:=
H_{t+1}(d)-H_t(d)
=
\frac{1}{(d-t)(d-t-1)}
\log\frac{d!}{t!(t+1)^{d-t}}.
\]
For an integer $1\le h<n$, define
\[
M_h:=\sum_{j=1}^h D_{n+1-j}-h.
\]
Let
\[
\Gamma
:=
\bigl\{
\Delta_t(d_i):
1\le i\le n-h,\
0\le t\le \min\{h-1,d_i-2\}
\bigr\},
\]
where $\Gamma$ is treated as a multiset. Write its elements in non-increasing order as $\lambda_1\ge \lambda_2\ge \cdots\ge \lambda_{|\Gamma|}$, and set $q:=\min\{M_h,|\Gamma|\}$.

\begin{theorem} \label{main_thm}
Let $G$ be a bipartite graph with vertex classes
$V=\{v_1,\ldots,v_n\}$ and $W=\{w_1,\ldots,w_n\}$. Denote
$d_i=\deg(v_i)$ and $D_i=\deg(w_i)$, and assume $1\leq d_1\le d_2\le \cdots\le d_n$, $1\leq D_1\le D_2\le \cdots\le D_n$. Then, for every $1\le h<n$,
\[
|\mathcal M_{\mathrm{perf}}(G)|
\le
h!
\prod_{i=1}^{n-h}(d_i!)^\frac{1}{d_i}
\prod_{r=1}^{q}2^{\lambda_r}.
\]
\end{theorem}

Theorem~\ref{main_thm_his1} is the case $h=1$ of
Theorem~\ref{main_thm}. The parameter $h$ gives a way to choose how many
vertices are revealed last. Larger values of $h$ may lead to better
bounds when these terminal vertices are highly constrained by the rest
of the graph.

In Example \ref{ex1}, we construct a family of graphs for which Theorem~\ref{main_thm} gives the exact number of perfect matchings, while Br\'egman's bound is exponentially larger. 

We next consider $C_4$-free bipartite graphs. Here, a graph is said to be $C_4$-free if it does not contain any cycle of length $4$. In \cite{araujo2022maximum}, Araujo, Balogh and Wang proved that if $G$ is a $C_4$-free balanced bipartite graph with $2n$ vertices and $n+k$ edges, then
\[
|\mathcal M_{\mathrm{perf}}(G)|
\le
2^{\frac{k}{3}}.
\]
This bound is sharp for disjoint unions of $6$-cycles. It is an
edge-count bound, while our estimate below depends more explicitly on
the local degree structure.

Let $G$ be a bipartite graph as stated in Theorem \ref{main_thm}. Suppose further that $G$ is $C_4$-free. Take $L\subseteq [n]:=\{1,2,\ldots,n\}$ to be the set of indices of the neighbours of $w_n$ in $V$. Define the multiset
\[
\Gamma'
:=
\bigl\{
\Delta_t(d_i):
i\in [n]\setminus L,\
0\le t\le \min\{D_n-1,d_i-2\}
\bigr\}.
\]
Write the elements of $\Gamma'$ in non-increasing order as $\mu_1\ge\mu_2\ge\cdots\ge\mu_{|\Gamma'|}$. Define
\[
M':=\sum_{j=1}^{D_n}D_{n+1-j}-2D_n+1,
\qquad
q':=\min\{M',|\Gamma'|\}.
\]

\begin{theorem} \label{main_thm_c4}
Let $G$ be a $C_4$-free bipartite graph as stated above. Then
\[
|\mathcal M_{\mathrm{perf}}(G)|
\le
\prod_{i\in [n]\setminus L}(d_i!)^{\frac{1}{d_i}}
\prod_{r=1}^{q'}2^{\mu_r}.
\]
\end{theorem}

In Example \ref{ex2}, we show graphs such that Theorem~\ref{main_thm_c4} gives the exact number of perfect matchings, while the Araujo--Balogh--Wang bound is exponentially larger. 

Moreover, the Araujo--Balogh--Wang bound is sharp for the disjoint union of $C_6$'s. Our Theorem \ref{main_thm_c4} is worse at first glance. However, an application to each connected component still gives the exact number of perfect matchings (See Remark \ref{remark_C6free} for more details).

\section{Preliminaries}\label{section2}

A random variable $X$ is a function from a probability space $(\Omega, \mathcal F, \mathbb P)$ to some measurable space $(\Omega', \mathcal F')$. For a measurable map $f$ from $(\Omega', \mathcal F')$ to $(\Omega'', \mathcal F'')$, the random variable $f(X)$ is defined to be the composition map $f\circ X$.

In this paper, we assume that all random variables take only finitely many values, and all logarithms are base $2$. For convenience, we adopt $0 \log 0 = 0$. For a random variable $X$, let $\range(X)$ denote the set of values $X$ takes with positive probability, and define the entropy
\[
H(X) := -\sum_{x\in \range(X)} \mathbb P(X=x) \log \mathbb P(X=x).
\]

For a pair of random variables $X, Y$, the conditional entropy of $X$ given $Y$ is defined by
\begin{equation} \label{eq_condition_entropy}
H(X\mid Y) := \sum_{y\in \range(Y)} \mathbb P(Y=y) H(X\mid Y=y),
\end{equation}
where
\[
H(X\mid Y=y) = - \underset{x\in \range(X)}{\sum} \mathbb P(X=x\mid Y=y) \log \mathbb P(X=x\mid Y=y).
\]

The following are basic properties of entropy.

\begin{lemma} \label{lem_entropy_basic}
Let $X,Y,X_1,\ldots,X_n$ be random variables. The following statements hold.

(\romannumeral1) $H(X)\le \log |\range(X)|$. The equality holds if and only if $X$ is uniformly distributed on $\range(X)$.

(\romannumeral2) $H(X_1,\cdots,X_n)=\sum_{i=1}^n H(X_i\mid X_1,\cdots,X_{i-1})$.

(\romannumeral3) $H(X\mid Y)\le H(X)$.
\end{lemma}

Indeed, Lemma \ref{lem_entropy_basic}(\romannumeral1) also holds for conditional distributions, which leads to the following.

\begin{lemma} \label{lem_conditional_log}
Let $X, Y$ be random variables, and $f$ be a measurable function from $\range(Y)$ to $[1,+\infty)$. Suppose that, for any $y\in \range(Y)$, the random variable $X$ can take at most $f(y)$ distinct values given that $Y=y$ occurs. Then
\[
H(X|Y) \leq \mathbb E \log f(Y) .
\]
\end{lemma}

\begin{proof}
Applying Lemma \ref{lem_entropy_basic}(\romannumeral1) to the conditional distribution of $X$ given $Y=y$, one obtains that
\[
H(X|Y=y) \leq \log f(y),\qquad y\in \range(Y).
\]
By \eqref{eq_condition_entropy},
\[
H(X|Y) \leq \sum\limits_{y\in \range(Y)} \mathbb P(Y=y)\log  f(y) = \mathbb E \log f(Y).
\]
\end{proof}

Recall that 
\[
H_t(d) := \frac{1}{d-t} \log \frac{d!}{t!},\qquad 0\le t\le d-1.
\]

For $H_0(d)$ and $H_1(d)$, we have the following inequality.

\begin{lemma}\label{ineq_2}
When $2 \le d_i \le d_j$,
\[
H_0(d_i) + H_1(d_j) \le H_0(d_j) + H_1(d_i)
\]
Equivalently, \(\Delta_0(d) = H_1(d)-H_0(d)\) is non-increasing in d for \(d\ge 2\).
\end{lemma}

\begin{proof}\label{pf_1}
It suffices to prove that
\[
H_1(d) - H_0(d) \ge H_1(d+1) - H_0(d+1)
\]
for $d \ge 2$, which is equivalent to
\[
\frac{\log(d!)}{d(d-1)} \ge \frac{\log((d+1)!)}{(d+1)d},
\]
i.e.,
\begin{equation} \label{eq_dfraction_square}
(d!)^2 \ge (d+1)^{d-1}.
\end{equation}

Note that
\[
(d!)^2 = 1^2 \cdot \left(\prod_{i=1}^{d-1}(1+i)\right)^2  = \prod_{i=1}^{d-1} (1+i)(d+1-i)  \ge \prod_{i=1}^{d-1} (d+1),
\]
since $(1+i)(d+1-i) = (d+1) + i(d-i) \ge (d+1)$.
So \eqref{eq_dfraction_square} holds, and the proof is complete.
\end{proof}

\section{A General Terminal-Set Entropy Framework}

In the rest of the paper, we always view a perfect matching $\sigma$ of the graph $G$ as a permutation in $S_n$, i.e., a bijection $\sigma:[n] \rightarrow [n]$,  where $\sigma(i) = j$ if and only if $v_iw_j \in \sigma$. Without confusion, we also write $\sigma=(\sigma(1),\cdots, \sigma(n))$.


We will work under a terminal-set framework. Let $L_\sharp \subseteq [n]$ denote the set of vertices to be revealed last, and write $L_\flat=[n]\setminus L_\sharp$ and  $h=|L_\sharp|$. 
For $i\in [n]$, set $\mathcal{N}_i := \{j : v_i w_j \in E(G)\}$, which records the neighbours of $v_i$ in $W$, and also denote $t_i(\sigma):=|\mathcal N_i\cap \sigma(L_\sharp)|$. 
Moreover, for a set $J\subseteq [n]$ of indices, we write $V(J)=\{v_j:\, j\in J\}$ and $W(J)=\{w_j:\, j\in J\}$. 

\begin{lemma} \label{lem_terminal_set}
Assume that $\mathcal M_{\mathrm{perf}}(G)\neq\emptyset$, and let $\Sigma$ be a random variable that is uniformly distributed on $\mathcal M_{\mathrm{perf}}(G)$. With the notation above,
\[
\log |\mathcal M_{\mathrm{perf}}(G)| \le \frac{1}{|\mathcal M_{\mathrm{perf}}(G)|}
\sum_{\sigma\in\mathcal M_{\mathrm{perf}}(G)}
\sum_{k\in L_\flat}
H_{t_k(\sigma)}(d_k) + H\big(\Sigma|_{L_\sharp}\, \big|\, \Sigma|_{L_\flat}\big).
\]
\end{lemma}

\begin{proof}
Let $\Sigma$ be a random variable that is uniformly distributed on $\mathcal{M}_{\mathrm{perf}}(G)$. By Lemma \ref{lem_entropy_basic}(\romannumeral1), one has
\[
H(\Sigma) = \log |\mathcal{M}_{\mathrm{perf}}(G)|.
\]

We reveal the vertices in $V(L_\flat)$ first, one by one in a uniformly random order, and reveal the vertices in $V(L_\sharp)$ last. Denote by
\[
\widetilde S_n := \{\tau \in S_n : \tau_j \in L_\flat \text{ for } 1\le j \le n-h, \ \tau_j \in L_\sharp \text{ for } n-h+1\le j \le n\},
\]
which has $h!(n-h)!$ elements. For $\tau\in \widetilde S_n$, write $\tau_j=\tau(j)$ $(j=1,2,\ldots,n)$ for simplicity.

For every $\tau\in\widetilde S_n$, Lemma \ref{lem_entropy_basic}(\romannumeral2) gives
\[
H(\Sigma)
=
\sum_{j=1}^{n-h}
H\bigl(
\Sigma(\tau_j)
\mid
\Sigma(\tau_1),\ldots,\Sigma(\tau_{j-1})
\bigr)
+
H\bigl(
\Sigma|_{L_\sharp}
\mid
\Sigma|_{L_\flat}
\bigr).
\]
Averaging over $\tau\in\widetilde S_n$ and defining
\[
\mathcal H_\flat
:=
\frac{1}{h!(n-h)!}
\sum_{\tau\in\widetilde S_n}
\sum_{j=1}^{n-h}
H\bigl(
\Sigma(\tau_j)
\mid
\Sigma(\tau_1),\ldots,\Sigma(\tau_{j-1})
\bigr),
\]
we obtain
\[
H(\Sigma)
=
\mathcal H_\flat
+
H\bigl(
\Sigma|_{L_\sharp}
\mid
\Sigma|_{L_\flat}
\bigr).
\]

Changing the order of the sum leads to
\[
\mathcal H_\flat
:=
\frac{1}{h!(n-h)!}
\sum_{\tau\in\widetilde S_n}
\sum_{k\in L_\flat} H(\Sigma(k) \mid \Sigma(\tau_1), \cdots, \Sigma(\tau_{j-1})),
\]
where $j:= \tau^{-1}(k)$ depends on both $k$ and $\tau$. 

Recall that $\mathcal N_i$ records the neighbours of $v_i$ in $W$. For $k\in L_\flat$ and $\tau\in  \widetilde S_n$, define
\[
M_{k,\sigma,\tau}:=\bigl|\mathcal N_k\setminus\{\sigma(\tau_1),\ldots,\sigma(\tau_{j-1})\}\bigr|,\qquad \sigma\in \mathcal M_{\mathrm{perf}}(G),
\]
and let $M_{k,\Sigma,\tau}:=\bigl|\mathcal N_k\setminus\{\Sigma(\tau_1),\ldots,\Sigma(\tau_{j-1})\}\bigr|$ be the corresponding random variable. Then, by Lemma \ref{lem_conditional_log},
\[
H\bigl(\Sigma(k)\mid \Sigma(\tau_1),\ldots,\Sigma(\tau_{j-1})\bigr)
\le
\mathbb E\log M_{k,\Sigma,\tau}.
\]

Now we expand the expectation over the uniform matching variable $\Sigma$, and group the terms by the value of $M_{k,\sigma,\tau}$, obtaining
\[
\mathbb{E}  \log M_{k,\Sigma,\tau} = \sum\limits_{\sigma\in \mathcal M_{\mathrm{perf}}(G)} (\log M_{k,\sigma,\tau}) \, \mathbb P(\Sigma =\sigma)= \sum\limits_{i=1}^{d_k} \sum\limits_{\sigma\in \mathcal M_{\mathrm{perf}}(G)\atop M_{k,\sigma,\tau}=i} \frac{\log i}{|\mathcal M_{\mathrm{perf}}(G)|}.
\]
And then
\[
\mathcal H_\flat \leq \frac{1}{h!(n-h)!}\frac{1}{|\mathcal M_{\mathrm{perf}}(G)|} \sum_{k\in L_\flat} \sum\limits_{\sigma \in \mathcal M_{\mathrm{perf}}(G)}  \sum\limits_{i=1}^{d_k} (\log i) \, \big|\{\tau \in \widetilde S_n:\, M_{k,\sigma,\tau}=i\}\big|.
\]

At this moment, let $k$ and $\sigma$ be given. Recall $t_k(\sigma):=|\mathcal N_k\cap \sigma(L_\sharp)|=|\sigma^{-1}(\mathcal N_k)\cap L_\sharp|$. Write $m_k(\sigma) := |\sigma^{-1}(\mathcal N_k)\cap L_\flat|=d_k-t_k(\sigma)$. Note that $m_k(\sigma)\geq 1$, since $k\in \sigma^{-1}(\mathcal N_k)\cap L_\flat$. 

If $k$ has rank $r$ in the relative order induced by $\tau$ on $\sigma^{-1}(\mathcal N_k)\cap L_\flat$, then exactly $r-1$ elements of this set precede $k$. Their images under $\sigma$ are already occupied neighbours of $v_k$. The number of available neighbours is then
\[
M_{k,\sigma,\tau} = t_k(\sigma) + (m_k(\sigma) - r +1) = d_k - r +1.
\]
Since the first $n-h$ positions are a uniform permutation of $L_\flat$, the relative rank $r$ is uniform in $\{1,\dots,m_k(\sigma)\}$. Therefore,
\[
\frac{1}{h!(n-h)!}  \big|\{\tau\in\widetilde S_n:\, M_{k,\sigma,\tau}=i\}\big|=\frac{1}{m_k(\sigma)} = \frac{1}{d_k - t_k(\sigma)}
\]
for $i=t_k(\sigma)+1,\dots,d_k$, and equals $0$ for $i=1,2,\ldots,t_k(\sigma)$. It follows that
\[
\mathcal H_\flat
\le
\frac{1}{|\mathcal M_{\mathrm{perf}}(G)|}
\sum_{k\in L_\flat}
\sum_{\sigma\in\mathcal M_{\mathrm{perf}}(G)}\sum_{i=t_k(\sigma)+1}^{d_k}
\frac{\log i}{d_k-t_k(\sigma)}  = \frac{1}{|\mathcal M_{\mathrm{perf}}(G)|}
\sum_{\sigma\in\mathcal M_{\mathrm{perf}}(G)}
\sum_{k\in L_\flat}
H_{t_k(\sigma)}(d_k),
\]
recalling that
\[
H_t(d)=\frac{1}{d-t}\log\frac{d!}{t!}.
\]
This completes the proof.
\end{proof}

\section{Proof of Theorems \ref{main_thm_his1} and \ref{main_thm}}

In the following, we first show the proof of the general case. 

\begin{proof} [Proof of Theorem \ref{main_thm}]
If $G$ has no perfect matching, then the asserted inequality is trivial. Hence we assume that $\mathcal M_{\mathrm{perf}}(G)\neq\emptyset$.

Let $\Sigma$ be a random variable that is uniformly distributed on $\mathcal{M}_{\mathrm{perf}}(G)$. Take $L_\flat=\{1,2,\ldots,n-h\}$ and $L_\sharp = \{n-h+1,\cdots,n\}$. With the same notations as in Lemma \ref{lem_terminal_set}, we obtain that
\[
\log |\mathcal M_{\mathrm{perf}}(G)| \le \frac{1}{|\mathcal M_{\mathrm{perf}}(G)|}
\sum_{\sigma\in\mathcal M_{\mathrm{perf}}(G)}
\sum_{k=1}^{n-h}
H_{t_k(\sigma)}(d_k) + H\big(\Sigma|_{L_\sharp}\, \big|\, \Sigma|_{L_\flat}\big).
\]

Regarding the second term on the right-hand side, assume that $\Sigma|_{L_\flat}=\sigma_\flat$ is known, for some injection $\sigma_\flat:\, L_\flat\rightarrow [n]$, i.e., the first $n-h$ vertices are revealed. Then $\Sigma|_{L_\sharp}$ can only take values in the set of bijections from $L_\sharp$ to $[n]\setminus \sigma_\flat(L_\flat)$, which has cardinality $h!$. Hence, by Lemma \ref{lem_conditional_log},
\[
H\big(\Sigma|_{L_\sharp}\, \big|\, \Sigma|_{L_\flat}\big) \le \mathbb{E} \log h! = \log h!.
\]

As for the first term, for convenience, denote it as $\mathcal{H}_\flat$. At this stage, let $\sigma\in \mathcal M_{\mathrm{perf}}(G)$ be fixed. Recall that $t_k(\sigma)=|\mathcal N_k\cap \sigma(L_\sharp)|$. In view of $\sigma(k)\in \mathcal N_k\cap \sigma(L_\flat)$, we have $0\leq t_k(\sigma)\leq \min\{h, d_k-1\}$. Noticing that $t_k(\sigma)$ counts the number of neighbours of $v_k$ from $W(\sigma(L_\sharp))$, the sum $\sum_{k=1}^{n-h} t_k(\sigma)$ counts the total number of edges between $V(L_\flat)$ and $W(\sigma(L_\sharp))$. Since $W(\sigma(L_\sharp))$ consists of $h$ vertices in $W$, their total degree is at most the sum of the largest $h$ degrees in $W$. Moreover, there are at least $h$ edges between $V(L_\sharp)$ and $W(\sigma(L_\sharp))$,
namely the $h$ matching edges used by $\sigma$. Hence
\[
\sum_{k=1}^{n-h}t_k(\sigma)\le
\sum_{j=1}^h D_{n+1-j}-h
=
M_h.
\]
Consequently,
\[
\mathcal{H}_\flat
\le
\max\left\{
\sum_{k=1}^{n-h}H_{t_k}(d_k):
0\le t_k\le \min\{h,d_k-1\},\
\sum_{k=1}^{n-h}t_k\le M_h
\right\}.
\]

For every $k$, recalling $\Delta_t(d) = H_{t+1}(d)-H_t(d)$. We always have $\Delta_t(d)> 0$ for $0\le t\le d-2$, since
\[
\frac{d!}{t!} =
\prod_{j=t+1}^{d}j > (t+1)^{d-t} \qquad (0\le t\le d-2).
\]
Now
\[
H_{t_k}(d_k)
=
H_0(d_k)+\sum_{s=0}^{t_k-1}\Delta_s(d_k),
\]
where the second sum is empty if $t_k=0$. Since $\sum\nolimits_{k=1}^{n-h}t_k\le M_h$, at most $M_h$ marginal increments can be chosen in total. The increments associated with a feasible vector
$(t_k)_{k=1}^{n-h}$ satisfy additional prefix constraints, but dropping
these constraints can only increase the maximum. Therefore their total
is at most the sum of the largest $q$ elements of $\Gamma$. Hence
\[
\mathcal{H}_\flat
\le
\sum_{k=1}^{n-h}H_0(d_k)
+
\sum_{r=1}^q\lambda_r,
\]
with $q=\min\{M_h,|\Gamma|\}$ and $\lambda_r$ $(r=1,\ldots,q)$ being the largest possible increments from
the multiset $\Gamma$. Combining the above upper bounds together, we conclude that
\[
\log |\mathcal M_{\mathrm{perf}}(G)|
\le
\log h!
+
\sum_{k=1}^{n-h}H_0(d_k)
+
\sum_{r=1}^q\lambda_r.
\]
The proof is completed.
\end{proof}

Now we prove the special case with explicit bound. 

\begin{proof} [Proof of Theorem \ref{main_thm_his1}]
Applying Theorem \ref{main_thm} with $ h=1 $, we have $ M_1 = D_n - 1 $. Thus the multiset of marginal increments is $\Gamma = \{\Delta_0(d_i):1\le i\le n-1,\ d_i\ge2\}$. Thus $\Gamma$ consists precisely of the increments $\Delta_0(d_i)$ corresponding to the degrees $d_i>1$. 
By Theorem \ref{main_thm}, the upper bound is
\[
|\mathcal M_{\mathrm{perf}}(G)| \le \prod_{i=1}^{n-1}(d_i!)^{\frac{1}{d_i}} \prod_{r=1}^{q} 2^{\lambda_r},
\]
where $ q = \min\{D_n-1, |\Gamma|\} $, and $ \lambda_1 \ge \lambda_2 \ge \cdots \ge \lambda_q $ are the largest $ q $ elements in $ \Gamma $.

By definition, we have $ \Delta_0(d)> 0 $ for all $ d \ge 2 $. Furthermore, Lemma \ref{ineq_2} shows that $ \Delta_0(d) $ is non-increasing as the degree $ d $ increases.
Since the degrees are sorted as $ d_1 \le d_2 \le \cdots \le d_{n-1} $, the elements in $\Gamma$ correspond exactly to the degrees starting from $d_\ell$, where $d_\ell$ is the first degree greater than $1$. These increments naturally appear in non-increasing order. Therefore, the largest $q$ increments are $\Delta_0(d_i)$ for $i = \ell, \ell+1, \cdots, p$, where $p = \min\{D_n + \ell - 2, n-1\}$. (If $\ell > p$, this index range is empty, meaning no increments are selected.) 

Thus, the sum of the selected top increments yields:
\[
\sum_{r=1}^{q} \lambda_r = \sum_{i=\ell}^{p} \Delta_0(d_i) = \sum_{i=\ell}^{p} (H_1(d_i) - H_0(d_i)).
\]

Replacing this into the upper bound, and noting that $(d_i!)^{\frac{1}{d_i}} = 1$ for $i < \ell$, we get
\begin{align*}
|\mathcal M_{\mathrm{perf}}(G)| &\le \prod_{i=1}^{n-1} (d_i!)^{\frac{1}{d_i}} \prod_{i=\ell}^{p} 2^{H_1(d_i) - H_0(d_i)} \\
 &= \prod_{i=1}^{\ell-1} 2^{H_0(d_i)} \prod_{i=\ell}^{p} 2^{H_1(d_i)} \prod_{i=p+1}^{n-1} 2^{H_0(d_i)} \\
&= \prod_{i=\ell}^{p} (d_i!)^{\frac{1}{d_i-1}} \prod_{i=p+1}^{n-1} (d_i!)^{\frac{1}{d_i}}.
\end{align*}

This completes the proof.
\end{proof}

\begin{remark}
For general $h\ge2$, a feasible vector $(t_k)$ may contain entries $t_k\ge2$, and hence higher-order increments $\Delta_1,\Delta_2,\ldots$ may occur. Unlike the increments $\Delta_0(d)$, these higher-order increments do not admit a uniform global ordering across different degrees.

Whether $\Delta_0(d_2)$ is larger than $\Delta_1(d_1)$ depends intricately on the specific values of $d_1$ and $d_2$. For instance, when $d_1 = d_2 \ge 3$, we have $\Delta_0(d_2) > \Delta_1(d_1)$; but for a fixed $d_1$, since $\Delta_0(d_2) \to 0$ as $d_2 \to \infty$, we eventually have $\Delta_0(d_2) < \Delta_1(d_1)$ for sufficiently large $d_2$.

Because of this crossover behavior, the largest $M_h$ elements in $\Gamma$ will be a complex mixture of $\Delta_0, \Delta_1, \cdots$ depending completely on the exact degree distribution. Therefore, one cannot specify a fixed, formulaic partitioning index to explicitly assign the exponents for the factorials as we did for $h=1$.
\end{remark}

Now we end this section with the following example.

\begin{example} \label{ex1}
Let $2\le h\le n/2$. We construct a bipartite graph $F_{n,h}$ with vertex classes $V=A\cup B\cup U$ and
\[
W=\{x_1,y_1,\ldots,x_h,y_h\}\cup\{z_1,\ldots,z_{n-2h}\},
\]
where
\[
A=\{a_1,\ldots,a_h\},\qquad
B=\{b_1,\ldots,b_{n-2h}\},\qquad
U=\{u_1,\ldots,u_h\}.
\]
Thus both vertex classes have size $n$. The edges of $F_{n,h}$ are defined as follows:
\[
N(a_i)=\{x_i,y_i\}, \quad (1\le i\le h),\qquad N(b_j)=\{z_j\}, \quad (1\le j\le n-2h),
\]
and every vertex $u_i\in U$ is adjacent to every vertex of $W$.

We first determine the number of perfect matchings. Every edge $b_jz_j$ is forced. For each $i\in[h]$, the vertex $a_i$ may be matched to either $x_i$ or $y_i$, and these choices are independent because the sets $\{x_i,y_i\}$ are pairwise disjoint. After these choices have been
made, exactly one vertex from each pair $\{x_i,y_i\}$ remains unmatched. Thus exactly $h$ vertices of $W$ remain, and they can be matched arbitrarily to the $h$ vertices of $U$. Consequently,
\[
|\mathcal M_{\mathrm{perf}}(F_{n,h})| = 2^h h!.
\]

We now apply Theorem~\ref{main_thm} with the vertices in $U$ chosen as the terminal vertices. The degrees on the $V$-side are $\deg(b_j)=1$, $\deg(a_i)=2$ and $\deg(u_i)=n$. Every vertex of $W$ has degree $h+1$: it has exactly one neighbour in $A\cup B$ and is adjacent to all $h$ vertices of $U$. Hence $D_1=\cdots=D_n=h+1$ and 
\[
M_h = \sum_{j=1}^{h}D_{n+1-j}-h = h(h+1)-h = h^2.
\]

The degree-$1$ vertices contribute no marginal increments. For each degree-$2$ vertex $a_i$, the only possible increment is $\Delta_0(2) = H_1(2)-H_0(2) =  \frac12$. Hence
\[
\Gamma = \Bigg\{
\underbrace{\frac12,\ldots,\frac12}_{h\text{ copies}}
\Bigg\}, \qquad q=\min\{h^2,h\}=h.
\]
Theorem~\ref{main_thm} therefore gives
\[
|\mathcal M_{\mathrm{perf}}(F_{n,h})| 
\le
h! \prod_{v\in A\cup B}(\deg(v)!)^{\frac{1}{\deg(v)}}
\prod_{r=1}^{h}2^{\frac{1}{2}}  = h!\,(2!)^{\frac{h}{2}}\,2^{\frac{h}{2}}  = 2^h h!.
\]
Thus our upper bound is exact for $F_{n,h}$.

In contrast, Br\'egman's theorem gives
\[
|\mathcal M_{\mathrm{perf}}(F_{n,h})| \le (2!)^{\frac{h}{2}}(n!)^{\frac{h}{n}} = 2^{\frac{h}{2}}(n!)^{\frac{h}{n}}.
\]
Since the sequence $(m!)^{1/m}$ is increasing and $n\ge 2h$, we have $(n!)^{h/n}\ge ((2h)!)^{1/2}$. It follows that the ratio between Br\'egman's bound and our bound is at least
\[
\frac{2^{\frac{h}{2}}\sqrt{(2h)!}}{2^h h!} = \sqrt{\frac{\binom{2h}{h}}{2^h}} > 1
\]
for every $h\ge2$. In particular, when $n=2h$, Stirling’s formula gives the asymptotic term $\frac{2^{h/2}}{(\pi h)^{1/4}}$ as $h\to\infty$. Thus Br\'egman's bound is exponentially larger than our bound as $h\to\infty$, whereas Theorem~\ref{main_thm} gives the exact number of perfect matchings.
\end{example}

\section{Proof of Theorem \ref{main_thm_c4}}
Now we complete the proof for $C_4$-free bipartite graphs. 

\begin{proof} [Proof of Theorem \ref{main_thm_c4}]
As in the previous proof, we assume without loss of generality that $\mathcal{M}_{\mathrm{perf}}(G)\neq \emptyset$. Let $\Sigma$ be a random variable that is uniformly distributed on $\mathcal{M}_{\mathrm{perf}}(G)$. Take $L_\flat=[n]\setminus L$ and $L_\sharp=L$, where $L$ is the index set of the neighbours of $w_n$. Then $V(L_\sharp)=N(w_n)$ and $h=|L_\sharp|=D_n$. By Lemma \ref{lem_terminal_set}, we have
\[
\log |\mathcal M_{\mathrm{perf}}(G)|
\le
\frac{1}{|\mathcal M_{\mathrm{perf}}(G)|}
\sum_{\sigma\in\mathcal M_{\mathrm{perf}}(G)}
\sum_{k\in L_\flat}
H_{t_k(\sigma)}(d_k)
+
H\big(\Sigma|_{L_\sharp}\mid \Sigma|_{L_\flat}\big).
\]

We first estimate the contribution from the vertices in $L_\flat$. For a fixed perfect matching $\sigma$, the quantity $\sum_{k\in L_\flat}t_k(\sigma)$ counts the number of edges between $V(L_\flat)$ and $W(\sigma(L_\sharp))$. Since $W(\sigma(L_\sharp))$ consists of $D_n$ vertices in $W$, the total number of edges incident to these vertices is at most the sum of the largest $D_n$ degrees in $W$, namely $\sum_{j=1}^{D_n}D_{n+1-j}$. 

We now estimate how many of these edges lie inside $V(L_\sharp)\times W(\sigma(L_\sharp))$. Since $w_n$ is adjacent to every vertex in $V(L_\sharp)$, it contributes exactly $D_n$ such edges. Moreover, every vertex $w\in W(\sigma(L_\sharp))\setminus\{w_n\}$ 
has at least one neighbour in $V(L_\sharp)$, because its matched partner under $\sigma$ belongs to $V(L_\sharp)$. Since $G$ is $C_4$-free, the vertices $w$ and $w_n$ have at most one common neighbour. Hence $w$ has exactly one neighbour in $V(L_\sharp)$. Therefore, the number of edges between $V(L_\sharp)$ and $W(\sigma(L_\sharp))$ is exactly $D_n+(D_n-1) = 2D_n-1$. Consequently,
\[
\sum_{k\in L_\flat}t_k(\sigma) \le \sum_{j=1}^{D_n}D_{n+1-j}-(2D_n-1) = M'.
\]

Therefore, with the same argument as in the proof of Theorem \ref{main_thm},
\[
\frac{1}{|\mathcal M_{\mathrm{perf}}(G)|}
\sum_{\sigma\in\mathcal M_{\mathrm{perf}}(G)}
\sum_{k\in L_\flat}
H_{t_k(\sigma)}(d_k)
\le
\sum_{k\in L_\flat}H_0(d_k)
+
\sum_{r=1}^{q'}\mu_r ,
\]
where $\mu_r$ are the largest $q'$ elements of $\Gamma'$.

It remains to estimate the terminal entropy term. Fix a value $\sigma_\flat\in \range(\Sigma|_{L_\flat})$. After the matching on $V(L_\flat)$ is revealed, the unmatched vertices of $W$ are precisely $W([n]\setminus \sigma_\flat(L_\flat))$. Since no vertex outside $V(L_\sharp)=N(w_n)$ is adjacent to $w_n$, we have $w_n\in W([n]\setminus \sigma_\flat(L_\flat))$. 
Now consider a vertex $w\in W([n]\setminus \sigma_\flat(L_\flat))\setminus\{w_n\}$. In any completion of the partial matching $\sigma_\flat$, the vertex $w$ must be matched to a vertex in $V(L_\sharp)$. Hence $w$ has at least one neighbour in $V(L_\sharp)$. Since $G$ is $C_4$-free and $w$ and $w_n$ already have the common neighbour structure described above, this neighbour is unique. 
Therefore, after $\sigma_\flat$ is fixed, all matches involving the vertices in $W([n]\setminus \sigma_\flat(L_\flat))\setminus\{w_n\}$ are forced. The remaining vertex of $V(L_\sharp)$ must then be matched to $w_n$. Hence the completion of $\sigma_\flat$ to a perfect matching is unique. Consequently, $H\big(
\Sigma|_{L_\sharp}
\mid
\Sigma|_{L_\flat}=\sigma_\flat
\big)=0$. Averaging over $\sigma_\flat$ gives $H\big(
\Sigma|_{L_\sharp}
\mid
\Sigma|_{L_\flat}
\big)=0$. 

Combining the above estimates, we obtain
\[
\log |\mathcal M_{\mathrm{perf}}(G)|
\le
\sum_{k\in L_\flat}H_0(d_k)
+
\sum_{r=1}^{q'}\mu_r .
\]
Since $L_\flat=[n]\setminus L$, the conclusion follows by exponentiating both sides.
\end{proof}

\begin{example} \label{ex2}
Fix an integer $r\ge 2$. Let $F_r$ be the bipartite graph with vertex classes $V=\{x,u_1,\ldots,u_r\}$, $W=\{y,z_1,\ldots,z_r\}$, and edge set
\[
E(F_r)
=
\{u_i y,\ u_i z_i,\ xz_i:\, 1\le i\le r\}.
\]
Equivalently, $F_r$ is the union of $r$ internally disjoint paths $x z_i u_i y$ of length $3$ between $x$ and $y$.

The graph $F_r$ is $C_4$-free. Indeed, for $i\ne j$, the vertices $u_i$ and $u_j$ have the unique common neighbour $y$, while $x$ and $u_i$ have the unique common neighbour $z_i$. Thus no two vertices in $V$ have more than one common neighbour.

We first count the perfect matchings of $F_r$. In any perfect matching, the vertex $y$ must be matched to one of the vertices $u_i$. Once $yu_i$ is chosen, every vertex $u_j$ with $j\ne i$ must be matched to $z_j$, and consequently $x$ must be matched to $z_i$. Conversely, each choice of $i\in[r]$ gives a perfect matching. Therefore
\[
|\mathcal M_{\mathrm{perf}}(F_r)|=r.
\]

We now apply Theorem~\ref{main_thm_c4}. Choose $w_n=y$. Then $D_n=\deg(y)=r$ and $L:=\{i\in [r]:\, v_i\in N(y)\}$. Thus $x$ is the only vertex of $V$ outside $V(L)$, and $d(x)=r$. The degrees of the vertices in $W$ are $\deg(y)=r$ and $\deg(z_i)=2$ for $1\le i\le r$. Hence the sum of the largest $r$ degrees in $W$ is $r+2(r-1)=3r-2$, and therefore $M' = (3r-2)-2r+1 = r-1$.

Since $x$ is the only vertex outside $V(L)$, the increment multiset is $\Gamma'= \{\Delta_0(r),\Delta_1(r),\ldots,\Delta_{r-2}(r)\}$. Thus
\[
|\Gamma'|=r-1,
\qquad
q'=\min\{M',|\Gamma'|\}=r-1.
\]
Theorem~\ref{main_thm_c4} now gives
\[
|\mathcal M_{\mathrm{perf}}(F_r)|  \le (r!)^{\frac{1}{r}} \prod_{t=0}^{r-2}2^{\Delta_t(r)} = 2^{H_0(r)+\sum_{t=0}^{r-2}\Delta_t(r)} = 2^{H_{r-1}(r)} = 2^{\log r} = r.
\]
Thus the upper bound in Theorem~\ref{main_thm_c4} is exact for $F_r$.

On the other hand, $F_r$ has $r+1$ vertices in each part and $3r$ edges. Hence the bound of Araujo, Balogh, and Wang gives
\[
|\mathcal M_{\mathrm{perf}}(F_r)|
\le
2^{\frac{3r-(r+1)}{3}}
=
2^{\frac{2r-1}{3}}.
\]
For $r=2$, the two bounds agree. For every $r\ge3$, one has $r<2^{(2r-1)/3}$. Indeed, if $R_r:=\frac{2^{(2r-1)/3}}{r}$, then $R_2=1$ and
\[
\frac{R_{r+1}}{R_r} = 2^{\frac{2}{3}}\,\frac{r}{r+1}>1
\qquad (r\ge2).
\]
Therefore our bound is strictly better for every $r\ge3$, and the ratio $\frac{2^{(2r-1)/3}}{r}$ grows exponentially with $r$.
\end{example}

\begin{remark} \label{remark_C6free}
For disconnected graphs, Theorem~\ref{main_thm_c4} may be substantially stronger when applied separately to each connected component. For example, let $G$ be the disjoint union of $s$ copies of $C_6$. A direct
application of Theorem~\ref{main_thm_c4} to the whole graph gives $|\mathcal M_{\mathrm{perf}}(G)| \le 2^{(3s-1)/2}$, whereas the Araujo--Balogh--Wang bound gives the sharp value $2^s$. However, Theorem~\ref{main_thm_c4} gives the exact bound $2$ for each individual copy of $C_6$. Since the number of perfect matchings is multiplicative over connected components, applying our theorem componentwise also gives
\[
|\mathcal M_{\mathrm{perf}}(G)|\le 2^s.
\]
The loss in the direct global application occurs because the terminal set $L=N(w_n)$ lies in only one component, while the other components are treated through their Br\'egman-type contributions.
\end{remark}

\section*{Acknowledgments}
The authors are grateful to Hao Chen for helpful discussions.

\nocite{*}
\bibliographystyle{plain}
\bibliography{reference}

@article{Radhakrishnan1997AnEP,
  author    = {J. Radhakrishnan},
  title     = {An Entropy Proof of {Bregman}'s Theorem},
  journal   = {Journal of Combinatorial Theory, Series A},
  volume    = {77},
  pages     = {161--164},
  year      = {1997},
  url       = {https://api.semanticscholar.org/CorpusID:26508239}
}

@article{bregman1973some,
  author    = {L. M. Br{\`e}gman},
  title     = {Some properties of nonnegative matrices and their permanents},
  journal   = {Doklady Akademii Nauk},
  volume    = {211},
  number    = {1},
  pages     = {27--30},
  year      = {1973},
  publisher = {Russian Academy of Sciences}
}

@article{cutler2011entropy,
  author    = {J. Cutler and A. J. Radcliffe},
  title     = {An entropy proof of the {Kahn-Lov{\'a}sz} theorem},
  journal   = {The Electronic Journal of Combinatorics},
  pages     = {P10--P10},
  year      = {2011}
}

@article{schrijver1978short,
  author    = {A. Schrijver},
  title     = {A short proof of {Minc}'s conjecture},
  journal   = {Journal of Combinatorial Theory, Series A},
  volume    = {25},
  number    = {1},
  pages     = {80--83},
  year      = {1978},
  publisher = {Elsevier}
}

@book{alon2016probabilistic,
  author    = {N. Alon and J. H. Spencer},
  title     = {The probabilistic method},
  year      = {2016},
  publisher = {John Wiley \& Sons}
}

@article{minc1963upper,
  author    = {H. Minc},
  title     = {Upper bounds for permanents of $\left(0,1\right)$-matrices},
  journal   = {Bulletin of the American Mathematical Society},
  volume    = {69},
  number    = {6},
  pages     = {789--792},
  year      = {1963},
  doi       = {10.1090/S0002-9904-1963-11031-9},
  mrnumber  = {MR0155843},
  url       = {https://doi.org/10.1090/S0002-9904-1963-11031-9}
}

@article{alon2008maximum,
  author    = {N. Alon and S. Friedland},
  title     = {The Maximum Number of Perfect Matchings in Graphs with a Given Degree Sequence},
  journal   = {The Electronic Journal of Combinatorics},
  volume    = {15},
  year      = {2008},
  month     = {04},
  doi       = {10.37236/888}
}

@book{Egorychev,
  author    = {G. Egorychev},
  title     = {Permanents},
  series    = {Series of Discrete Mathematics},
  publisher = {SFU},
  address   = {Krasnoyarsk},
  year      = {2007}
}

@article{friedland2008upper,
  author    = {S. Friedland},
  title     = {An upper 
  bound for the number of perfect matchings in graphs},
  journal   = {arXiv preprint arXiv:0803.0864},
  year      = {2008}
}

@article{araujo2022maximum,
  author    = {I. Araujo and J. Balogh and Y. Wang},
  title     = {Maximum determinant and permanent of sparse 0-1 matrices},
  journal   = {Linear Algebra and its Applications},
  volume    = {645},
  pages     = {194--228},
  year      = {2022},
  publisher = {Elsevier}
}
\end{document}